\documentclass[10pt]{article}
\usepackage{times,amsmath,amssymb,exscale,mathrsfs}
\usepackage{pst-all}
\usepackage{graphicx}
\usepackage{xcolor} % For colored text
\usepackage{epsfig}
\newtheorem{definition}{Definition}%[section]
\newtheorem{theorem}{Theorem}%[section]
\newtheorem{proposition}{Proposition}%[section]
\newtheorem{remark}{Remark}%[section]
\newtheorem{example}{Example}%[section]

\newtheorem{assumption}{Assumption}%[section]
\newenvironment{proof}{\noindent{\bf Proof:}}{$\Box$\medskip}
\newcommand{\C}{\mathbb{C}}

\newcommand{\g}{\mathfrak{g}}
\newcommand{\G}{\mathcal{G}}
\newcommand{\Ham}{\mathcal{H}}

\newcommand{\bk}{{\bf k}}

\newcommand{\R}{\mathbb{R}}

\newcommand{\W}{\mathcal{W}}
\newcommand{\Z}{\mathbb{Z}}
\newcommand{\zero}{\mathbf{0}}
\newcommand{\uno}{{\, 1\!\! 1\,}}
\newcommand{\sh}{{\,\scriptstyle \sqcup\!\sqcup\,}}

\title{Computing normal forms and formal invariants of dynamical systems by means of word series}
\author{A. Murua\footnote{Konputazio Zientziak eta A.\ A.\  Saila, Informatika
 Fakultatea, UPV/EHU, E--20018 Donostia--San Sebasti\'{a}n,  Spain. Email: Ander.Murua@ehu.es}
\  and J.M. Sanz-Serna\footnote{(Corresponding author) Departamento de  Matem\'aticas, Universidad Carlos III de Madrid, Legan\'es (Madrid), Spain.
 Email: jmsanzserna@gmail.com}}
\date{\today}
\begin{document}
\maketitle
\begin{center}
{\em Dedicated to Juan Luis V\'azquez on his 70th birthday}
\end{center}
\bigskip
\begin{abstract}
We show how to use extended word series in the reduction of continuous and discrete dynamical systems to normal form and in the computation of formal invariants of motion in Hamiltonian systems. The manipulations required  involve complex numbers rather than vector fields or diffeomorphisms. More precisely we construct a  group $\overline{\G}$ and a Lie algebra $\overline{\g}$ in such a way that the elements of $\overline{\G}$ and $\overline{\g}$ are families of complex numbers; the operations to be performed involve the multiplication $\bigstar$ in $\overline{\G}$ and the bracket of
$\overline{\g}$ and result in {\em universal} coefficients that are then applied to write the normal form or the invariants of motion of the specific problem under consideration.
\end{abstract}

\medskip

\noindent{\bf Keywords and sentences:} Word series, extended word series, B-series, words, Hopf algebras, shuffle algebra, Lie groups, Lie algebras, Hamiltonian problems, integrable problems, normal forms, resonances.
\medskip


\section{Introduction}

In this paper we show how to use extended word series in the reduction of continuous and discrete dynamical systems to normal form and in the computation of formal invariants of motion in Hamiltonian systems of differential equations. The manipulations required in our approach involve complex numbers rather than vector fields or diffeomorphisms. More precisely, we construct a  group $\overline{\G}$ (semidirect product of the additive group of $\C^d$ and the group of characters of the shuffle Hopf algebra) and a Lie algebra $\overline{\g}$ in such a way that the elements of $\overline{\G}$ and $\overline{\g}$ are families of complex numbers; the operations to be performed involve the multiplication $\bigstar$ in $\overline{\G}$ and the bracket of
$\overline{\g}$ and result in {\em universal} coefficients that are then applied to write the normal form or the invariants of motion of the specific problem under consideration.\footnote{In fact it is possible to present $\overline{\G}$ and $\overline{\g}$ in terms of a {\em universal property} in the language of category theory. We shall not be concerned with that task here.}

The present approach originated from our earlier work on the use of formal series to analyze numerical integrators; see \cite{china} for a survey of that use. In a seminal paper, Hairer and Wanner \cite{HW} introduced the concept of B-series as a means to perform systematically the manipulations required to investigate the accuracy of Runge-Kutta and related numerical methods for ordinary differential equations. B-series are series of functions; there is a term in the series associated with each rooted tree. The letter B here refers to John Butcher, who in the early 1960's enormously simplified,  through a  book-keeping system based on rooted trees, the task of Taylor expanding Runge-Kutta solutions. This task as performed by Kutta and others before John Butcher's approach was extraordinarily complicated and error prone. Key to the use of B-series is the fact that the substitution of a B-series in another B-series  yields a third B-series, whose coefficients may be readily found by operating with the corresponding rooted trees and are independent of the differential equation being integrated. In this way the set of coefficients itself may be endowed with a product operation and becomes the so-called Butcher's group. The set of B-series resulting from a specific choice of differential equations is then a homomorphic image of the Butcher group. It was later discovered (see eg \cite{brouder}) that the Butcher group was not a mere device to understand numerical integrators but an important mathematical construction (the group of characters of a Hopf algebra structure on the set of rooted trees) which has found applications in several fields, notably in renormalization theories and noncommutative geometry. In \cite{part1} and \cite{part2} B-series found yet another application outside numerical mathematics when they were employed to average systems of differential equation with periodic or quasiperiodic forcing.

Our work here is based on {\em word} series  \cite{words}, an alternative to B-series defined in \cite{orlando}, \cite{part3} (see also \cite{anderfocm}, \cite{part2}). Word series are parameterized by the words of an alphabet $A$ and are more compact than the corresponding B-series parameterized by rooted trees with coloured nodes.  Section 2 provides a review of the use of word series. The treatment there focuses on the presentation of the rules that apply when manipulating the series in practice and little attention is paid to the more algebraic aspects. In particular the material in Section 2 is narrowly related to standard constructions involving the shuffle Hopf algebra over the alphabet $A$ (see \cite{words} and its references); we avoid to make this connection explicit and prefer to give a self-contained elementary exposition.
Section 3 presents the class of perturbed problems investigated in the paper; several subclasses are discussed in detail, including nonlinear perturbations of linear problems and perturbations of integrable problems written in action/angle variables \cite{arnoldmec}. In order to account for the format (unperturbed + perturbation) being considered, we employ what we call {\em extended word series}. In the particular case of action/angle integrable problems, extended word series have already been introduced in \cite{words}; here we considerably enlarge their scope. Section 4, that is parallel to Section 2, studies the rules that apply to the manipulation of extended words series. Again many of those rules essentially appear in \cite{words}, but the {\em ad hoc} proofs we used there do not apply in our more general setting so that new, more geometric proofs are required. The main results of the paper are presented in the final Section 5. We first address the task of reducing differential systems to normal forms via changes of variables (Theorem~\ref{th:theoremnormal}); as pointed out above, both the normal form and the  change of variables are written in terms of scalar coefficients that may be easily computed. For Hamiltonian problems the normal form is Hamiltonian and the change of variables symplectic. We also describe in detail (Theorem~\ref{th:new}) the freedom that exists when finding the normal form/change of variables. By going back to the original variables, we find a decomposition of the given vector field as a commuting sum of two fields: the first generates a flow that is conjugate to the unperturbed problem and the second accounts for the effects of the perturbations that are not removable by conjugation. This decomposition is the key to the construction of formal invariants of motion in Hamiltonian problems. We provide very simple recursions for the computation of the coefficients that are required to write down the decomposition of the vector field and the invariants of motion. Finally we briefly outline a parallel theory for discrete dynamical systems.

Some of the results in Section 5 have precedents in the literature. The reduction to normal form (but not the investigation of the associated freedom) appears in \cite{words} but only the particular setting of action/angle variables (ie of Example~3 in Section~\ref{ss:examples}). Decompositions of the field as a commuting sum features in \cite{part2}, but only for a class of problems much narrower than the one we deal with here. That paper also presents, in a more restrictive scenario, recursions for the computation of invariants of motion. However the methodology in \cite{part2} is different from that used here and the recursions found there do not coincide with those presented in this work. The reference \cite{fm} is very relevant. It works in the restricted context of Example~1 in Section~\ref{ss:examples} and  emphasizes  the role played by different Hopf algebras, in particular the shuffle Hopf algebra and the commutative Hopf algebra of rooted trees with colored vertices, and the connection with Ecalle's mould calculus. Studied in detail is  the reduction of vector fields and diffeomorphisms to their linear parts with the help of  series of linear differential operators rather than of word series (the relation between both kinds of series has been discussed in
\cite{words}).

It should be pointed out that, just as the notion of word series may be modified to define extended word series adapted to perturbed problems, it is both possible and interesting (cf \cite{fm}) to consider {\em extended B-series}, a task that we plan to address in future contributions.

All the developments here operate with formal series of smooth maps. In order not to clutter the exposition we shall omit throughout the qualifiers \lq formal\rq\ and \lq smooth\rq. It is of course possible, after truncating the expansions, to derive bounds as in \cite{orlando}, \cite{part3}, but such a task is out of our scope. The  groups used here may be turned into Lie groups modeled in  Fr\'{e}chet spaces, see \cite{geir}.

\section{Word series}
\label{sec:words}

In this section we briefly review the use of  word series, a tool that is essential for the work in this paper. Proofs and additional details may be seen in \cite{words}.

\subsection{Definition of word series}
Let $A$ be a  finite or  infinitely countable set of indices (the alphabet) and assume that, for each $\ell\in A$, $f_\ell:\C^D \to \C^D$  is a  map. With each nonempty word $\ell_1\cdots \ell_n$ made with letters from $A$ we associate a {\em word basis function}. These functions are recursively defined by
\begin{equation*}%\label{eq:wbf}
f_{\ell_1\cdots \ell_n}\!(x) =  f^\prime_{\ell_2\cdots \ell_n}\!(x)\,f_{\ell_1}\!(x), \quad n >1
\end{equation*}
($f^\prime_{\ell_2\cdots \ell_n}\!(x)$ denotes the value at $x$ of the Jacobian matrix of $f_{\ell_2\cdots \ell_n}$). For the empty word, the associated  basis function is the identity map $x\mapsto x$. We  denote by $\W$ the set of all words (including the empty word $\emptyset$) and by $\C^\W$ the vector space of all the mappings $\delta: \W\to \C$; the notation $\delta_w$ will be used to refer to the value that $\delta$ takes at the word $w\in W$.

Given the maps $f_\ell$, $\ell\in A$, with each $\delta\in\C^\W$  we associate the formal series
$$
W_\delta(x) = \sum_{w\in \W} \delta_w f_w(x),
$$
and say that $W_\delta(x)$ is the {\em word series} with coefficients $\delta$. It is also possible to consider the real case with $f_\ell:\R^D \to \R^D$ and $\delta\in\R^\W$.

As an example, consider the nonautonomous initial-value problem
\begin{equation}\label{eq:ode}
\frac{d}{dt} x= \sum_{\ell\in A} \lambda_\ell(t) f_\ell(x),\qquad x(0) = x_0,
\end{equation}
where the $\lambda_\ell$ are given  scalar-valued functions of $t$. Its solution may be represented as $x(t) = W_{\alpha(t)}(x_0)$,
with the coefficients $\alpha(t)$ given by the iterated integrals
\begin{equation}\label{eq:alpha}
\alpha_{\ell_1\cdots \ell_n}\!(t) = \int_0^t dt_n\,\lambda_{\ell_n}\!(t_n)\int_0^{t_n}dt_{n-1}\,\lambda_{\ell_{n-1}}\!(t_{n-1})\cdots \int_0^{t_2} dt_1\,\lambda_{\ell_1}\!(t_1).
\end{equation}
The series $W_{\alpha(t)}(x_0)$ is essentially the Chen-Fliess series used in control theory.

\subsection{Operations with word series}

Given $\delta,\delta^\prime\in\C^\W$, their  convolution product $\delta\star\delta^\prime\in\C^\W$ is defined by
\begin{equation*}%\label{eq:convol}
(\delta\star\delta^\prime)_{\ell_1\cdots \ell_n} = \delta_\emptyset\delta^\prime_{\ell_1\cdots \ell_n}
+ \sum_{j=1}^{n-1} \delta_{\ell_1\cdots \ell_j}\delta^\prime_{\ell_{j+1}\cdots \ell_n}
+\delta_{\ell_1\cdots \ell_n}\delta^\prime_\emptyset
\end{equation*}
($(\delta\star\delta^\prime)_\emptyset = \delta_\emptyset\delta^\prime_\emptyset$). This product is not commutative, but it is associative and has a unit (the element $\uno\in \C^\W$ with $\uno_\emptyset = 1$ and $\uno_w = 0$ for $w\neq \emptyset$).

If $w$ and $w^\prime$ are words, their shuffle product will be denoted by  $w\sh w^\prime$.
 The set $\G$ consists of those $\gamma \in \C^\W$  that satisfy the  shuffle relations: $\gamma_\emptyset = 1$ and, for each $w,w^\prime\in \W$,
\begin{equation*}
\gamma_w\gamma_{w^\prime} = \sum_{j=1}^N \gamma_{w_j}\qquad \mbox{\rm if}\qquad w\sh w^\prime = \sum_{j=1}^N w_j.
\end{equation*}
 This set is a group for the operation $\star$.
For $\gamma\in\G$, $W_\gamma(x)$ may be substituted in an arbitrary word series $W_\delta(x)$, $\delta\in\C^\W$, to get a new word series whose coefficients are
given by the convolution product $\gamma\star\delta$:
\begin{equation}\label{eq:act}
W_\delta\big (W_{\gamma}(x)\big) = W_{\gamma\star \delta}(x).
\end{equation}

We denote by $\g$ the set of elements $\beta\in\C^\W$ such that $\beta_\emptyset = 0$ and for each pair of nonempty words $w,w^\prime$,
\[
\sum_{j=1}^N \beta_{w_j} = 0\qquad \mbox{\rm if}\qquad w\sh w^\prime = \sum_{j=1}^N w_j.
\]
With  the skew-symmetric convolution bracket defined by
\begin{equation*}
[\beta,\beta^\prime] = \beta\star\beta^\prime-\beta^\prime\star\beta,
\end{equation*}
 $\g$ is a Lie algebra and, in fact, it may be seen as the Lie algebra of  $\G$ if this is seen as a Lie group \cite{geir}. The elements in $\G$ and $\g$ are related by the usual exponential and logarithmic power series (where, of course, powers are taken with respect to the product $\star$). For $\beta\in\g$ and arbitrary $\delta\in \C^\W$, the product $\beta*\delta$ has the following word series interpretation:
 \[
 W_{\beta*\delta}(x) =  W_\delta^\prime(x) W_\beta(x).
 \]
This implies in particular that the convolution bracket just defined corresponds to the Lie-Jacobi bracket of the associated word series:
\[
 W^\prime_{\beta^\prime}(x) W_\beta(x) -  W^\prime_\beta(x) W_{\beta^\prime}(x)= W_{[\beta,\beta^\prime]}(x),\qquad \beta,\beta^\prime\in \g.
\]
For $\beta\in\g$ the Dynkin-Specht-Wever formula  may be used to rewrite the word series in terms of iterated commutators:
\begin{equation}\label{eq:dynkin}
W_\beta(x)
 = \sum_{n=1}^\infty \frac{1}{n} \sum_{\ell_1,\dots,\ell_n\in A}
 \beta_{\ell_1\cdots \ell_n} [[\cdots[[f_{\ell_1},f_{\ell_2}],f_{\ell_3}]\cdots],f_{\ell_n}](x).
\end{equation}
(For $n=1$ the terms in the inner sum are of the form $\beta_{\ell_1} f_{\ell_1}(x)$.)

If $\beta(t)\in\g$ for each real $t$, the initial value problems
\begin{equation}\label{eq:ode3}
\frac{d}{dt} x(t) = W_{\beta(t)}(x(t)),\qquad x(0) =x_0,
\end{equation}
may be solved formally by using the ansatz $x(t) = W_{\alpha(t)}(x_0)$. In fact, in view of (\ref{eq:act}), we may write
$$
\frac{d}{dt}W_{\alpha(t)}(x_0) = W_{\beta(t)}(W_{\alpha(t)}(x_0)) = W_{\alpha(t)\star\beta(t)}(x_0),\qquad W_{\alpha(0)}(x_0) = x_0,
$$
which leads to a linear, nonautonomous initial value problem in $\G$
\begin{equation}\label{eq:odebeta}
\frac{d}{dt} \alpha(t) = \alpha(t)\star\beta(t),\qquad \alpha(0) = \uno.
\end{equation}
  that may be solved by successively determining $\alpha_w(t)$, $w\in\W_n$, $n = 0, 1,\dots$  For each $t$ the element $\alpha(t)\in\C^\W$ found in this way belongs to the group $\G$.
(The solvability of (\ref{eq:odebeta}) is referred to as the regularity of $\G$ in Lie group terminology.) Conversely, any   curve $\alpha(t)$ of group elements with $ \alpha(0) = \uno$ solves a problem of the form (\ref{eq:odebeta}) with
$$
\beta(t) = \alpha(t)^{-1} \star \frac{d}{dt} \alpha(t).
$$
A change of variables $x = W_{\kappa}(X)$, $\kappa\in\G$, transforms the problem (\ref{eq:ode3}) into
$$
\frac{d}{dt} X(t) = W_{B(t)}(X(t)), \qquad X(0) = X_0,
$$
with $B(t) = \kappa \star\beta(t)\star \kappa^{-1}$, $W_{\kappa}(X_0) = x_0$.

Consider finally the particular case where the dimension $D$  is even and each $f_\ell(x)$ is a Hamiltonian vector field, i.e.\ $f_\ell(x) = J^{-1}\nabla H_\ell(x)$, where $J^{-1}$ is the standard symplectic matrix. In view of the correspondence that exists between Hamiltonian fields/Lie-Jacobi brackets and Hamiltonian functions/Poisson brackets, for each $\beta\in\g$, the vector field $W_\beta(x)$ is Hamiltonian,
$$
W_\beta(x) = J^{-1}\nabla\Ham_\beta(x)
$$
and its  Hamiltonian function is (see (\ref{eq:dynkin}))
$$
\Ham_\beta(x) = \sum_{w\in\W,\, w\neq\emptyset} \beta_w H_w(x),
$$
where, for each nonempty word $w=\ell_1\cdots \ell_n$,
\begin{equation}\label{eq:wordham}
H_w(x) = \frac{1}{n}\{\{\cdots\{\{H_{\ell_1},H_{\ell_2}\},H_{\ell_3}\}\cdots\},H_{\ell_n}\}(x)
.
\end{equation}
 (We assume that the Poisson bracket has been defined in such a way that the vector field generated by the Poisson bracket of two Hamiltonians is the Lie-Jacobi bracket of the fields corresponding to Hamiltonians; in the literature it is more frequent to use as Poisson bracket the opposite of the one employed here \cite{arnoldmec}.)

If $\beta,\beta^\prime \in\g$, the Poisson bracket of $\Ham_\beta$ and $\Ham_{\beta^\prime}$ may be expressed in terms of the convolution bracket of the coefficients as $\Ham_{[\beta,\beta^\prime]}$.

For Hamiltonian systems, changes of variables $x = W_{\kappa}(X)$, $\kappa\in\G$, are canonically symplectic; after the change of variables the system is again Hamiltonian and the new Hamiltonian function is obtained by changing variables in the old Hamiltonian function:
$$
H_\beta(W_\kappa(X)) = H_B(X),\qquad B = \kappa\star\beta*\star\kappa^{-1}.
$$

\section{Perturbed  problems}

\subsection{Preliminaries}
\label{ss:prelim}

In the remainder of the paper we shall be concerned with initial value problems
\begin{equation}
  \label{eq:odegf}
  \frac{d}{dt} x = g(x)+f(x) , \quad x(0)=x_0,
\end{equation}
where  $f,g:\C^D \to \C^D$  and $f$ can be decomposed as
\begin{equation}
\label{eq:fdec}
f(x) = \sum_{\ell \in A} f_\ell(x)
\end{equation}
for a set of indices $A$, referred to as the alphabet as in Section~\ref{sec:words}. We work under the assumption that,
for each $\ell \in A$, there is $\nu_\ell \in \C$ such that
\begin{equation}
  \label{eq:[f0fa]}
  [g,f_\ell] = \nu_\ell\, f_\ell;
\end{equation}
in other words, each vector field $f_\ell$ is an eigenvector of the operator $[g,\cdot]$ (adjoint of $g$), where $[\cdot,\cdot]$ represents the usual Lie-Jacobi bracket ($[g,f] = f^\prime g-g^\prime f$). Due to well-known properties of the Lie-Jacobi bracket,  (\ref{eq:[f0fa]}) is independent of the choice of coordinates: if an arbitrary  change of variables $x = x(X)$ is performed in (\ref{eq:odegf}), then (\ref{eq:[f0fa]}) holds in the $x$ variables if and only if it holds in the $X$ variables.

The next proposition gives an alternative formulation of the assumption;  $\varphi_t$ denotes the solution flow of $g$.

\begin{proposition}
\label{prop:1} Equation (\ref{eq:[f0fa]}) is equivalent to the requirement that
for   each $x\in\C^D$, and each real $t$,
\begin{equation}
  \label{eq:[f0fa]2}
  \varphi_t^\prime(x)^{-1}f_\ell(\varphi_t(x)) = e^{t\nu_\ell} f_\ell(x).
\end{equation}

\end{proposition}
\begin{proof} Assume that (\ref{eq:[f0fa]}) holds. If $C(t)$ is the left hand-side of (\ref{eq:[f0fa]2}), then
\begin{eqnarray*}
\frac{d}{dt} C(t) &=& - \varphi_t^\prime(x)^{-1} \left(\frac{d}{dt} \varphi_t^\prime(x)\right) \varphi_t^\prime(x)^{-1} f_\ell(\varphi_t(x))\\
 &&+\varphi_t^\prime(x)^{-1} f^\prime _\ell(\varphi_t(x)) \frac{d}{dt} \varphi_t(x)\\
&=&- \varphi_t^\prime(x)^{-1} g^\prime(\varphi_t(x)) \varphi_t^\prime(x)\: \varphi_t^\prime(x)^{-1}f_\ell(\varphi_t(x))\\
&& +\varphi_t^\prime(x)^{-1} f^\prime _\ell(\varphi_t(x)) g(\varphi_t(x)) \\
&=& \nu_\ell C(t)
\end{eqnarray*}
and integration leads to $C(t) = \exp(t\nu_\ell)C(0)$, ie to (\ref{eq:[f0fa]2}). Conversely, differentiation with respect to $t$ of (\ref{eq:[f0fa]2}) at $t=0$ results in (\ref{eq:[f0fa]}).
\end{proof}

Geometrically (\ref{eq:[f0fa]2}) says that, for each fixed $t$, the diffeomorphism $\varphi_t$ pulls back the vector field $f_\ell$ to the vector field $e^{t\nu_\ell} f_\ell$. In other words, $f_\ell$ is an eigenvector with eigenvalue $\exp(t\nu_\ell)$ of the linear operator that associates with each vector field $h$ its  pull back  $(\varphi_t^\prime(\cdot))^{{-1}}(h\circ\varphi_t)(\cdot)$.

In the applications we have in mind, the differential system in (\ref{eq:odegf}) is seen as a perturbation of the system $(d/dt) x = g(x)$ and the
flow $\varphi_t$ is considered to be known. If the solution $x(t)$ of (\ref{eq:odegf}) is sought in the form $x(t) = \varphi_t(y(t))$, then  $y(0) = x_0$ and, after invoking (\ref{eq:[f0fa]2}), it is trivially found that $y(t)$ satisfies
$$
\frac{d}{dt} y = \sum_{\ell \in A} e^{t\nu_\ell} f_\ell(y).
$$
Since this system is of the form (\ref{eq:ode}) with
\begin{equation}\label{eq:lambdaexp}
\lambda_\ell(t) = e^{t\nu_\ell},\qquad \ell\in A,
\end{equation}
we find that $y(t) = W_{\alpha(t)}(x_0)$, where, as we know, the coefficients $\alpha(t)$ are given by (\ref{eq:alpha}). We conclude that the solution of (\ref{eq:odegf}) has the representation $x(t) =\varphi_t(W_{\alpha(t)}(x_0))$.

\subsection{A class of perturbed problems}
\label{ss:examples}

In what follows we assume that the field $g$ in (\ref{eq:odegf}) lies in a finite-dimensional vector space of commuting vector fields (that is, an Abelian Lie algebra of vector fields), with  a basis $\{g_1,\ldots,g_d\}$ ($[g_j,g_k] = 0$ for $j\neq k$), and that each $f_\ell$ is an eigenvector of each operator $[g_j,\cdot]$. The  elements of the Abelian Lie algebra will be denoted by
\begin{equation}\label{eq:gmu}
g^v = \sum_{j=1}^{d} v_j\, g_j,
\end{equation}
where $v=(v_1,\ldots,v_d)$ is a constant vector.

   More precisely {\em we work hereafter under the following assumption}.

\begin{assumption}
\label{ass:2}
The system (\ref{eq:odegf}) is such that:
\begin{itemize}
\item  $f$ may be decomposed as in (\ref{eq:fdec}).

\item There are linearly independent, commuting vector fields $g_j$ and constants $v_j$, $j=1,\dots,d$, such that $g=g^v$ with $g^v$ as in (\ref{eq:gmu}).

 \item For each $j=1,\dots, d$ and each $\ell \in A$, there is $\nu_{j,\ell} \in \C$ such that
\begin{equation}
  \label{eq:1}
  [g_j,f_\ell] = \nu_{j,\ell}\, f_\ell.
\end{equation}
\end{itemize}
\end{assumption}

Clearly, (\ref{eq:1}) implies that (\ref{eq:[f0fa]}) holds with $\nu_\ell$ given by the ($v$-dependent) quantity
\begin{equation}\label{eq:nuel}
\nu^v_\ell = \sum_{j=1}^d v_j \nu_{j,\ell}.
\end{equation}
If we denote by $\varphi_{v}$ the flow at time $t=1$ of $g^v$, the $t$-flow  of $g^v$ coincides with $\varphi_{tv}$ and, according to
Section~\ref{ss:prelim}, the solution of (\ref{eq:odegf}) has the representation $x(t) = \varphi_{tv}(W_{\alpha(t)}(x_0))$, where the coefficients $\alpha(t)$ are given by (\ref{eq:alpha}), (\ref{eq:lambdaexp}), (\ref{eq:nuel}). Note that these coefficients {\em depend on $v$ and the $\nu_{j,\ell}$ but are otherwise independent of $g$ and $f$}.

Let us now examine some classes of differential systems that satisfy Assumption \ref{ass:2}.

\begin{example}\em
  Consider the case
  \begin{equation}\label{eq:L}
    \frac{d}{dt} x = L x + f(x),
  \end{equation}
  where $L$ is a diagonalizable $D \times D$  matrix, and   $f(x)$  is  polynomial, ie each of its components is a polynomial in the components of $x$. Let $\mu_1$, \dots, $\mu_d$
  denote the distinct nonzero eigenvalues of $L$, so that $L$ may be uniquely decomposed as
\begin{equation*}
  L = \mu_1 L_1 + \cdots + \mu_d L_d,
\end{equation*}
where the $D\times D$ matrices $L_1,\ldots,L_d$  are projectors ($L_j^2 = L_j$) with $L_jL_k = 0$ if $j\neq k$.
Thus (\ref{eq:gmu}) holds for $g_j(x) = L_j x$, $v_j = \mu_j$. Furthermore consider, for each $v=(v_1,\ldots,v_d)$, the
diffeomorphism $x\mapsto \exp(v_1 L_1 + \cdots + v_d L_d) x$. This pulls  $f$ back into
\begin{equation*}
e^{-(v_1 L_1 + \cdots + v_d L_d)} f(e^{v_1 L_1 + \cdots + v_d L_d} x),
\end{equation*}
an expression that
can be uniquely rewritten as a sum
\begin{equation*}
\sum_{(k_1,\ldots,k_d) \in  A} e^{k_1 v_1 + \cdots + k_d v_d} f_{(k_1,\ldots,k_d) }(x),
\end{equation*}
where $A$ is a finite subset of $\Z^d$ and, for each $\bk = (k_1,\ldots,k_d)$ in $A$, $f_\bk$ is a polynomial map. It follows that $f = \sum_\bk f_\bk$ and that
 the pull back of
 $f_\bk$ is $\exp(v_1 L_1 + \cdots + v_d L_d) f_\bk$. According to Proposition~\ref{prop:1}, (\ref{eq:1}) holds with
\begin{equation*}
\nu_{j,\bk} = k_j, \quad \bk \in A,\quad j=1,\ldots,d.
\end{equation*}
The case where the components of $f$ are power series in the components of $x$ may be treated similarly,
cf \cite{fm}. Analytic systems of differential equations having an equilibrium at the origin are of the form (\ref{eq:L}), provided that the linearization at the origin is diagonalizable; the perturbation $f$ then contains terms  of degree $\leq 2$ in the components of $x$.
\end{example}

\begin{example}\em
In some applications, including Hamiltonian mechanics, the system (\ref{eq:L}) possesses some symmetry that implies that the nonzero eigenvalues of $L$ occur in pairs $\pm \mu$, with $+\mu$ and $-\mu$ having the same multiplicity. Let us then assume that, in the preceding example, $d$ is even and the nonzero eigenvalues satisfy
$\mu_{d-j+1}=-\mu_{j}$ for $j= 1, \dots, d$. The matrix $v_1 L_1 + \cdots + v_d L_d$ considered above does not have for arbitrary choices of the parameters $v_j$ a symmetric spectrum and, in order to not loose the symmetry, one may consider an alternative way of satisfying Assumption~\ref{ass:2}. Specifically, we may take
\begin{equation*}
\bar g_j(x) = L_j x - L_{d+1-j}x, \quad j=1,\ldots,d/2,
\end{equation*}
the alphabet $\bar A$  obtained from $A$ through the formula
\begin{equation*}
\bar A = \{(k_1-k_{d},k_2-k_{d-1}, \ldots,k_{d/2}-k_{d/2+1})\ : \
(k_1,\ldots,k_d) \in A \subset \Z^{d/2}\},
\end{equation*}
and $\bar f_{(\bar k_1,\ldots, \bar k_{d/2})}$  given, for each  $(\bar k_1,\ldots,\bar k_{d/2}) \in A$, as the sum of all $f_{(k_1,\ldots,k_d)}$ such that
\begin{equation*}
 (\bar k_1,\ldots,\bar k_{d/2}) =(k_1-k_{d},k_2-k_{d-1}, \ldots,k_{d/2}-k_{d/2+1}).
\end{equation*}
\end{example}
\begin{example}\em
We now consider real systems of the form
$$
\frac{d}{dt} \left[ \begin{matrix}y\\ \theta\end{matrix}\right]
= \left[ \begin{matrix}0\\ \omega\end{matrix}\right]
+f(y,\theta),
$$
where $y\in\R^{D-d}$, $0<d\leq D$, $\omega\in\R^d$ is a vector of frequencies $\omega_j\neq 0$, $j = 1,\dots,d$, and $\theta$ comprises $d$ angles, so that $f(y,\theta)$ is $2\pi$-periodic  in each component of $\theta$ with Fourier expansion
%\begin{equation}\label{eq:fourier}
$$
f(y,\theta) = \sum_{\bk \in\Z^d} \exp(i \bk\cdot \theta)\: \hat f_\bk(y).
$$
%\end{equation}

After introducing the  functions
\begin{equation*}
f_\bk(y,\theta) = \exp(i\bk\cdot \theta)\: \hat f_\bk(y),\qquad y\in\R^{D-d},\: \theta\in\R^d,
\end{equation*}
the system takes the form (\ref{eq:odegf})--(\ref{eq:fdec}) with $x=(y,\theta)$, $A=\Z^d$, and
\begin{equation*}
g(y,\theta)
= \left[ \begin{matrix}0\\ \omega\end{matrix}\right].
\end{equation*}
If some $f_{\bk}(x)$ are identically zero, then  $A$ may of course be taken to be a subset of $\Z^d$.
Assumption~\ref{ass:2} holds with $v_j = \omega_j$ ($j=1,\ldots,d$), each $g_j(x)$  a constant unit vector, and
\begin{equation*}
\nu_{j,\bk} = i \, k_j
\end{equation*}
for each $j=1,\ldots,d$ and each $\bk =(k_1,\ldots,k_d) \in A$.
\end{example}

\begin{example}\em
 Assume that in (\ref{eq:odegf}), the dimension $D$ is even with $D/2-d=m\geq 0$ and that the vector of unknowns takes the form
$$
(p^1, \dots, p^m; a^1,\dots, a^d; q^1,\dots,q^m; \theta^1,\dots,\theta^d),
$$
where $p^j$ is the momentum  conjugate to the co-ordinate $q^j$ and $a^j$ is the momentum (action) conjugate to coordinate (angle) $\theta^j$. Consider the Hamiltonian function
\begin{equation}\label{eq:rennes1}
\sum_{j=1}^d \omega_j H_j(a) + H(p;a;q;\theta),\quad H_j(a) =  \omega_ja^j,\:j=1,\dots,d,
\end{equation}
where $H$ is $2\pi$-periodic in each of the components of $\theta$ and has  Fourier expansion
$$
H(p;a;q;\theta) = \sum_{\bk\in \Z^d} H_\bk(p;a;q;\theta);\quad  H_\bk =\exp(i\bk\cdot\theta)\hat H_\bk(p;a;q),\:\bk\in\Z^d.
$$
 We set $\nu_{j,\bk} = ik_j$; it is then an exercise to check that
\begin{equation*}
\{ H_j,H_k\} = 0,\: j\neq k,\qquad \{H_j,H_\bk\} = \nu_{j,\bk} H_\bk.
\end{equation*}
 If $g_j$ and $f_\bk$ denote the Hamiltonian vector fields corresponding to $H_j$ and $H_\bk$ respectively, we have
\begin{equation*}
[ g_j,g_k] = 0,\: j\neq k,\qquad [g_j,f_\bk] = \nu_{j,\bk} f_\bk,
\end{equation*}
so that Assumption~\ref{ass:2} holds for the Hamiltonian system associated with (\ref{eq:rennes1}). Of course this example is a particular instance of the one we just considered above.
\end{example}

If $v$ is a $d$-vector and $w=\ell_1\cdots \ell_n$ a word, we extend the notation introduced in (\ref{eq:nuel}) and set
\begin{equation*}
\nu^v_{w} =  \nu^v_{\ell_1} +\cdots +\nu^v_{\ell_n}= \sum_{j=1}^{d} v_j\, (\nu_{j,\ell_1}+ \cdots + \nu_{j,\ell_n})
\end{equation*}
($\nu^v_\emptyset = 0$). In the construction of normal forms, we shall consider, for each $d$-vector $v$,  the vector space $\mathcal{V}(v)$ of all $d$-vectors $u$ for which $\nu^u_w = 0$ whenever
$\nu^v_w= 0$, $w\in\W$. This vector space
 is useful to describe {\em resonances}, as we now illustrate in a particular case. In the situation described in Example 3, assume that $d=2$. The components $v_1$ and $v_2$ are the frequencies $\omega_1$, $\omega_2$, of the unperturbed motion, the letters are of the form
$\bk\in\Z^2$ and $\nu_w^v = i (s(w)_1 v_1+s(w)_2 v_2)$, where $s(w)_1\in\Z$ and $s(w)_2\in \Z$ denote the first and second components of the sum $s(w)\in \Z^2$ of the letters of the word
$w$. If $v_1$ and $v_2$ are rationally independent then $\nu_w^v= 0$ if and only if $s(w)_1 = s(w)_2 = 0$ and $\mathcal{V}(v)$ comprises all $2$-vectors. However if $v_2\neq 0$ and the quotient
$v_1/v_2$ is a rational number $p/q$, then $\nu_w^v= 0$ if and only if $(s(w)_1,  s(w)_2)$ is proportional to $(q,-p)$ and $\mathcal{V}(v)$ is the one-dimensional subspace spanned by
$(p,q)$. Generally speaking,  the presence of resonances results in  a decrease of the dimension of $\mathcal{V}(v)$; this in turn implies that in Section~\ref{s:normal} fewer invariant quantities will exist.

\begin{remark}\em
The relations $[g_j,f_\ell] = \nu_{j,\ell} f_\ell$ and $[g_k,f_\ell] = \nu_{k,\ell} f_\ell$ imply that for each pair of complex numbers $c_j$, $c_k$
$$[c_jg_j+c_kg_k,f_\ell] =
(c_j\nu_{j,\ell}+c_k\nu_{k,\ell})f_\ell.$$  It is then clear that if Assumption~\ref{ass:2} holds for the Abelian Lie algebra spanned by the $g_j$, $j=1,\dots,d$, then it also holds for all its linear subspaces (equivalently for all its Abelian Lie subalgebras). Moving to a subspace/subalgebra in Section~\ref{s:normal} will in general simplify the computations required to find normal forms at the expense of reducing the number of linearly independent invariant quantities. An instance of the possibility of moving to a subspace appeared above In Example 2.
\end{remark}

\section{Extended word series}

\subsection{The definition of extended word series}

The material in the preceding section suggests the following definition.

\begin{definition} Given the commuting vector fields $g_j$, $j = 1,\dots,d$, and the family of vector fields $f_\ell$, $\ell\in A$, with each $(v,\delta) \in \C^d\times \C^\W$ we associate its {\em extended word series:}
\[
\overline{W}_{(v,\delta)}(x) =  \varphi_{v}(W_{\delta}(x)).
\]
\end{definition}

In the particular case envisaged in Example 3 of Section~\ref{ss:examples} extended word series were introduced in \cite{words}; the definition given in this paper applies in the more general context of Assumption~\ref{ass:2}. Most of the properties of extended word series discussed in \cite{words} remain valid in the present general setting, but they require different proofs.

Now the solution of (\ref{eq:odegf})  may be compactly expressed as $x(t) = \overline W_{(tv,\alpha(t))}(x_0)$.

\subsection{The operators $\Xi_v$ and $\xi_v$}

For each $d$-vector $v$, we shall use  the linear operator  $\Xi_{v}$  in $\C^\W$ that maps each $\delta \in \C^\W$ into the element of $\C^\W$ defined by
\begin{equation*}%\label{eq:Xi}
(\Xi_{v}\delta)_w =\exp(\nu^v_w)\: \delta_w.
\end{equation*}
Similarly the linear operator $\xi_{v}$ on $\C^{\W}$ is defined by
\begin{equation*}
(\xi_{v} \delta)_{w} = \nu^v_w\: \delta_{w}.
\end{equation*}
Thus $\Xi_{v}$ and $\xi_{v}$ are  {\em diagonal} operators with eigenvalues $\exp(\nu^v_{w})$ and
 $\nu^v_w$ respectively.
Observe that
\begin{equation*}
 \frac{d}{dt} \Xi_{t v} =
 \Xi_{t v} \xi_{v} = \xi_{v}\Xi_{t v}.
 \end{equation*}
 The operator $\Xi_v$ maps $\G$ into $\G$ and $\g$ into $\g$. Furthermore $\Xi_v$ is a homomorphism for the convolution product:
 $\Xi_{v}(\gamma \star \delta) = (\Xi_{v} \gamma) \star (\Xi_{v} \delta)$ if $\gamma, \delta\in \C^{\W}$. This implies that
 $\exp_\star(\Xi_v\beta)=\Xi_v \exp_\star(\beta)$ for $\beta\in\g$ and that $\xi_v$ is a derivation: $\xi_v(\delta\star\delta^\prime) =
 (\xi_v \delta)\star\delta^\prime + \delta\star(\xi_v \delta^\prime)$ for $\delta,\delta^\prime\in\C^\W$.

The formulae in the next proposition will be used later.
 \begin{proposition}
\label{prop:2}
For each $\beta \in \g$, $v \in \C^d$,
\begin{align}
\label{eq:l2c}
[g^v,W_{\beta}](x) &= W_{\xi_v \beta}(x),\\
\label{eq:l2b}
\varphi^\prime_{v}(x)^{-1}   W_{\beta}(\varphi_{v}(x)) &=    W_{\Xi_{v}\beta}(x).
\end{align}
Furthermore, for each
$\gamma \in \G$, $v \in \C^d$,
\begin{equation}
\label{eq:l2a}
W_{\gamma}(\varphi_{v}(x)) =  \varphi_{v}(W_{\Xi_{v}\gamma}(x)),
\end{equation}
and
\begin{equation}\label{eq:12d}
(W^\prime_\gamma(x))^{-1} g^v(W_\gamma(x)) =  g^v(x) - (W^\prime_\gamma(x))^{-1} W_{\xi_v\gamma}(x).
\end{equation}
\end{proposition}
\begin{proof}  From the Jacobi identity and (\ref{eq:1})
\[
[g^v, [f_{\ell_1}, f_{\ell_2}]] = (\nu^v_{\ell_1}+\nu^v_{\ell_2}) [f_{\ell_1}, f_{\ell_2}] = \nu^v_{\ell_1\ell_2} [f_{\ell_1}, f_{\ell_2}],
\]
and, by induction, for the iterated commutator
\[
I _{\ell_1 \cdots\ell_n}=[[\cdots[[f_{\ell_1},f_{\ell_2}],f_{\ell_3}]\cdots],f_{\ell_n}]
\]
we find
\[
[g^v, I_{\ell_1\cdots\ell_n}] = \nu^v_{\ell_1\cdots\ell_n}I _{\ell_1 \cdots\ell_n}.
\]
Therefore (\ref{eq:l2c}) is a consequence of (\ref{eq:dynkin}). Proposition~\ref{prop:1} then establishes (\ref{eq:l2b}).

From
(\ref{eq:l2b}), the change of variables $x=\varphi_v(X)$ pulls back the vector field $W_\beta$ into the vector field $W_{\Xi_v \beta}$; for the corresponding  flows at $t=1$ we then have $W_{\exp_\star(\beta)}\circ \varphi_v = \varphi_v\circ W_{\exp_\star(\Xi_v\beta)}$, where, as pointed out above,
$W_{\exp_\star(\Xi_v\beta)}= W_{\Xi_v\exp_\star(\beta)}$. Since all elements $\gamma\in\G$ are of the form $\exp_\star(\beta)$, $\beta\in\g$, we have proved (\ref{eq:l2a}).

To obtain (\ref{eq:12d}), write (\ref{eq:l2a}) with $tv$ replacing $v$, differentiate with respect to $t$ and evaluate at $t=0$.
\end{proof}

\subsection{The group $\overline{\G}$}

The symbol $\overline{\G}$ denotes the set  $\C^d \times \G$.
For each $t$, the
solution coefficients $(tv,\alpha(t))$ found above provide an example of element of $\overline{\G}$.
For $(u,\gamma)\in\overline{\G}$ and $(v,\delta)\in\C^d\times \C^\W$ we set
\begin{equation*}
(u,\gamma) \bigstar(v,\delta) = (
v+\delta_\emptyset u,
\gamma \star \Xi_{u} \delta)\in \C^d\times \C^\W.
\end{equation*}
For this operation $\overline{\G}$ is a noncommutative group;
$(\C^d,\uno)$ and $(0,\G)$ are subgroups of $\overline{\G}$. The unit of $\overline{\G}$ is $ \overline{\uno} = (0,\uno)$.
Note that, for each $(u,\gamma) \in \overline{\G}$,
\begin{equation*}
(u,\gamma) = (0,\gamma) \bigstar (u,\uno), \qquad (u,\uno) \bigstar (0,\gamma) = (u,\Xi_{u} \gamma).
\end{equation*}
In particular $(0,\gamma)$ and  $(u,\uno)$ commute if $\Xi_u$ leaves $\gamma$ invariant.

By using (\ref{eq:act}) and (\ref{eq:l2a}), it is a simple exercise to check that
\begin{equation*}%\label{eq:act2}
\overline{W}_{(v,\delta)}\big(\overline{W}_{(u,\gamma)}(x)\big) =
\overline{W}_{(u,\gamma)\bigstar (v,\delta)}(x), \qquad \gamma, \delta\in{\G}.
\end{equation*}
In fact  the operation $\bigstar$ has been defined so as to ensure this composition rule.

\subsection{The Lie algebra $\overline{\g}$}
\label{sss:tlag}

As a set, the Lie algebra $\overline{\g}$ of the group $\overline{\G}$ consists of the elements
$(v,\beta)\in\C^d \times\g$;  by differentiating curves in $\overline{\G}$, it is found that the Lie bracket of two elements in $\overline{\g}$  has the expression
\begin{equation*}
[(v,\beta),(u,\eta))] =
%s \, \xi \eta - r\,  \xi \delta + \delta \star \eta -\eta \star \delta \in \g.
(0,\xi_{v} \eta - \xi_{u} \beta +  [\beta,	\eta]).
\end{equation*}
With each element $(v,\beta)$ of $\overline{\g}$ we associate the  vector field
$
  g^v(x) + {W}_{\beta}(x)
$.
From (\ref{eq:l2c}) in Proposition~\ref{prop:2} we find
for arbitrary  $(v,\beta),(u,\eta) \in \overline{\g}$,
\begin{equation*}
[g^v + {W}_{\beta},g^u + {W}_{\eta}] = {W}_{[(v,\beta),(u,\eta)]},
\end{equation*}
and therefore the correspondence between $\overline{\g}$ and vector fields preserves the Lie algebra structure.

\subsection{Differential equations in $\overline{\G}$}
\label{sss:eqnsoverlineG}
Consider the initial value problem
\begin{equation}
\label{eq:odebetab}
\frac{d}{dt} x(t) = g^v(x(t)) + W_{\beta(t)}(x(t)),\qquad x(0) =x_0,
\end{equation}
where $\beta(t)\in\g$ for each $t$, and $v \in \C^d$.
By using Proposition~\ref{prop:2}, it is found that a time-dependent change of variables
$x = \varphi_{t v}(z) = \overline{W}_{(tv,\uno)}(z)$,  transforms the problem into
$$
\frac{d}{dt} z(t) =  W_{\Xi_{t v} \beta (t)}(z(t)), \qquad z(0) = x_0,
$$
and thus, (\ref{eq:odebetab}) may be  solved as $x(t) = \varphi_{t v}(W_{\alpha(t)}(x))$, ie
$x(t)=\overline{W}_{(t v,\alpha(t))}(x_0)$, where $\alpha(t)$ is the solution of the initial value problem in $\G$
\begin{equation*}
\frac{d}{dt} \alpha(t) = \alpha(t) \star \Xi_{t v} \beta(t), \quad \alpha(0) = \uno.
\end{equation*}
This may be integrated  as we integrated (\ref{eq:odebeta}). It is of interest to point out that, after using the definition of $\bigstar$, it is easily checked that the function of $t$ $(tv,\alpha(t))$ found in this way is the solution of the following nonautonomous initial value problem in $\overline{\G}$
\[
\frac{d}{dt}(u, \alpha) =(u, \alpha) \bigstar (v,  \beta(t)), \quad (u(0),\alpha(0)) = \overline{\uno};
\]
this, in turn, is clearly analogous to the problem (\ref{eq:odebeta}) in $\G$.

By means of (\ref{eq:12d}) it may easily be shown that a change of variables
$x = \overline{W}_{(u,\kappa)}(X)$, $\kappa\in\G$,  transforms the differential equation in (\ref{eq:odebetab}) into
$$
\frac{d}{dt} X(t) =g^v(X(t)) +  W_{B(t)}(X(t)), ,
$$
where  $B(t)$ is determined from
\begin{equation*}
B(t) = \kappa *(\Xi_u\beta(t))*\kappa^{-1} - (\xi_{v} \kappa)*\kappa^{-1}.
\end{equation*}

The following result, to be used later, is easily checked and corresponds to the well-known fact that changes variables commute with the computation of Lie brackets of vector fields.
\begin{proposition}
\label{l:LieBracketInvariance}
Assume that $(u,\kappa)$ is an element in the group $\overline{\G}$. Let the elements $(v^1,\Delta^1),(v^2,\Delta^2),(v^3,\Delta^3) \in \overline{\g}$ be related to
the elements $(v^1,\delta^1),(v^2,\delta^2),(v^3,\delta^3) \in \overline{\g}$
through
\begin{equation*}
\Delta^j = \kappa \star (\Xi_{u}\delta^j ) \star \kappa^{-1} - ( \xi_{v^j} \kappa) \star \kappa^{-1}, \quad j =1,2,3.
\end{equation*}
Then $[(v^1,\Delta^1),(v^2,\Delta^2)]=(v^3,\Delta^3)$ if and only if $[(v^1,\delta^1),(v^2,\delta^2)]=(v^3,\delta^3)$.
\end{proposition}

\subsection{The exponential of an element in $\overline{\g}$}

The exponential of an element $(v,\beta) \in \overline{\g}$ is, according to the preceding material, $(v,\alpha(1))$, where
$\alpha(t)$ is found by solving
\begin{equation}\label{eq:odexibeta}
\frac{d}{dt} \alpha = \alpha \star \Xi_{tv} \beta,\qquad \alpha(0) = \uno.
\end{equation}
At variance with the case of the group $\G$ and its Lie algebra $\g$, where the exponential is a bijection from the Lie algebra to the group, there are elements in $\overline{\G}$ that are not the exponential of an element in $\overline{\g}$. However, inverting the exponential for a given $(v,\gamma) \in \overline{\G}$ is always possible provided that $v$ is such that $\exp(\nu_w^{v})\neq 1$ for each word $w \in \W$ for  which $\nu_w^{v}\neq 0$.

\begin{proposition}
\label{th:logarithm}
 Given $v \in \C^d$ such that
\begin{equation}
\label{eq:non-resonance-for-log}
\nu_w^{v} \neq  2 k \pi\, i \quad \mbox{for all} \quad k \in \Z\backslash\{0\} \quad \mbox{and}\quad w \in \W,
\end{equation}
 the restriction of the exponential to the set
$
 \{(v,\beta)\ :  \  \beta \in \g\} \subset \overline{\g}
$
is injective, and gives a bijection from this set to
$
\{(v,\gamma)\ :  \  \gamma \in \G\} \subset \overline{\G}$.
\end{proposition}
\begin{proof} Given $\gamma\in\G$, in order to find the logarithm of $(v,\gamma)$ we have to determine $\beta\in\g$ such that $\alpha(1) = \gamma$, where $\alpha$ satisfies
(\ref{eq:odexibeta}).
Assume that the values of  $\beta$ at all words with $<n$ letters have already been determined  and choose $w\in \W_n$. Then
\begin{equation*}
\frac{d}{dt} \alpha_w = \Big(\Xi_{tv} \beta\Big)_w + \dots = \exp(t\nu^{v}_w)\beta_w+\cdots,
\end{equation*}
where the dots stand for terms involving values of $\beta$ at words with $<n$ letters. Integration leads to an equation
\begin{equation*}
\gamma_w = \int_0^1 \exp(t\nu^{v}_w)\,dt\: \beta_w +\cdots
\end{equation*}
that may be solved uniquely for $\beta_w$ because, under the hypothesis of the proposition, the integral does not vanish.
\end{proof}

\subsection{Perturbed Hamiltonian problems}
If for each $j$, $g_j(x)$ is a Hamiltonian vector field with Hamiltonian function $H_j(x)$, and each $f_\ell(x)$  is a Hamiltonian vector field with Hamiltonian function $H_\ell(x)$, then for each $(v,\beta)\in\overline{\g}$, the  vector field $g^v(x) + {W}_{\beta}(x)$ is  Hamiltonian, with Hamiltonian function
\begin{equation}
\label{eq:Hamseries}
\mathcal{H}_{(v,{\beta})}(x) = \sum_{j=1}^{d} v_j \, H_j(x) +
\sum_{w\in\W}{\beta}_w\, H_w(x),
\end{equation}
where the $H_w(x)$ are as in (\ref{eq:wordham}). For any $(u,\kappa)\in \overline{\G}$, the  map $x \mapsto \overline{W}_{(u,\kappa)}(x)$ is canonically symplectic. When this canonical map is used to change variables in the Hamiltonian system with Hamiltonian function $\mathcal{H}_{(v,\beta)}$ the new system is Hamiltonian and its Hamiltonian function is found by composing the old Hamiltonian with the change of variables.

As we have illustrated in Example~4 above, in some cases the following additional assumption holds:
\begin{assumption}\label{ass:ham}
The Hamiltonian functions $H_j(x)$, $j=1,\dots,d$, $H_\ell(x)$, $\ell\in A$ are such that
$$
\{H_j,H_k\}= 0,\: j\neq k,\qquad \{H_j,H_\ell\} = \nu_{j,\ell} H_\ell.
$$
\end{assumption}
Clearly this assumption implies that the Hamiltonian system satisfies Assumption~\ref{ass:2}. We point out that, in addition, each $H_j$ is a conserved quantity for the Hamiltonian flow generated by $H_k$, $k\neq j$.

Under Assumption~\ref{ass:ham},
the Poisson bracket of two   Hamiltonian functions $\mathcal{H}_{(v,{\beta})}$, $\mathcal{H}_{(v^\prime,{\beta^\prime})}$  may be expressed by means of the bracket in $\overline{\g}$ as
$\mathcal{H}_{[(v,{\beta}),(v^\prime,{\beta^\prime})]}$. Furthermore
for the change of variables we have the formula
\begin{equation*}
  \mathcal{H}_{(v,\beta)}(\overline{W}_{(u,\kappa)}(X)) = \mathcal{H}_{(v,B)}(X),
\end{equation*}
with
$$
B = \kappa \star (\Xi_{u}\beta ) \star \kappa^{-1} - ( \xi_{v} \kappa) \star \kappa^{-1}.
$$

\section{Normal forms}
\label{s:normal}

Let us  present our main results.

\subsection{Normal forms for elements in the Lie algebra $\overline{\g}$}
We now  study autonomous systems of the form
\begin{equation}\label{eq:systnormal}
\frac{d}{dt}x = g^{v}(x) +
 W_{\beta}(x),\qquad \beta \in \g.
\end{equation}
This format yields the original problem (\ref{eq:odegf})--(\ref{eq:fdec}) in the particular case where  $\beta$ is chosen as
\begin{equation}\label{eq:b}
\beta_w = 1 \quad \mbox{if}\quad w\in\W_1,\qquad \beta_w = 0 \quad\mbox{if}\quad w\notin \W_1;
\end{equation}
other choices of $\beta$ \cite{words} are of interest eg in the analysis of numerical integrators for (\ref{eq:odegf})--(\ref{eq:fdec})  in relation to the so-called modified equations \cite{ssc}, \cite{hlw}.
Our aim is to change variables $x = W_\kappa(X) = \overline{W}_{(0,\kappa)}(X)$, $\kappa\in\G$, in order to simplify
(\ref{eq:systnormal}).
According to Section~\ref{sss:eqnsoverlineG},
in the new variables we have
\begin{equation}\label{eq:systnormaltrans}
\frac{d}{dt}X = g^{v}(X) +  W_{\widehat{\beta}}(X),
\end{equation}
with
\begin{equation}\label{eq:normaleqn}
\widehat{\beta} = \kappa *\beta*\kappa^{-1} - (\xi_{v} \kappa)*\kappa^{-1}.
\end{equation}

We  choose  ($v$-dependent) elements $\widehat{\beta}\in\g$ and $\kappa\in \G$ subject to (\ref{eq:normaleqn}) and such that
$\widehat{\beta}$ is as simple as possible; then the system is said to have been brought to normal form. The maximum simplification given by $\widehat{\beta} = 0$ can only be reached if the relations $\kappa\in\G$, $\xi_v\kappa = 0$ imply $\kappa = \uno$, ie if $\nu_w^v\neq 0$ for all nonempty words $w$. If  $\nu_w^v= 0$ for some nonempty words, there are parts of the perturbation that commute with $g^{v}$ and of course those  cannot be eliminated by means of a change of variables, see eg~\cite{arnoldode}. The aim is then to get $\widehat{\beta}\in\g$ such that
$[g^{v}, W_{\widehat{\beta}}]=0$ (rather than $W_{\widehat{\beta}}=0$). In terms of the coefficients of the series we demand
$[(v,0),(0,\widehat{\beta})]=(0,\xi_{v} \widehat{\beta})=0$, ie
\begin{equation}\label{eq:condnormal}
\nu^{v}_w\:\widehat{\beta}_{w} = 0,
\end{equation}
for each nonempty word $w$. Equivalently,
$ \widehat{\beta}_w = 0$  for all words $w$ such that $\nu^{v}_{w} \neq 0$.

We have the following result:

\begin{theorem}
\label{th:theoremnormal}
Given $(v,\beta) \in \overline{\g}$, there exists  $\kappa\in\G$ such that the element $\widehat{\beta}\in {\g}$ defined in (\ref{eq:normaleqn}) satisfies that $[(v,0),(0,\widehat{\beta})] = (0,\xi_{v} \widehat{\beta})=0$. In addition $(0,\widehat{\beta})$ commutes with all elements $(u,0)$, $u\in \mathcal{V}(v)$.

Therefore, the change of variables $x = \overline{W}_\kappa(X)$,  reduces the system (\ref{eq:systnormal})
to the normal form (\ref{eq:systnormaltrans}) in such a way that the vector fields
$g^{v}(X)$ and ${W}_{\widehat \beta}(X)$ commute and the solutions of (\ref{eq:systnormaltrans}) satisfy
\[
X(t) =  \varphi_{t {v}} (\hat \varphi_t(X(0))) = \hat\varphi_t (\varphi_{t {v}}(X(0))),
\]
where $\hat \varphi_t$ is the solution flow of the system $(d/dt)X = W_{\widehat{\beta}}(X)$. Moreover,  the vector field
  $W_{\widehat \beta}(X)$ commutes with $g^{u}(X)$ for all vectors $u\in\mathcal{V}(v)$.

Solutions of (\ref{eq:systnormaltrans})  possess the extended word series representation
$$X(t) = \overline{W}_{(t {v},\widehat{\alpha}(t))}(X(0)),$$where
\[
(t {v},\widehat{\alpha}(t)) = (t{v}, \exp_\star(t\widehat{\beta})) = (0,\exp_\star(t\widehat{\beta})) \bigstar (t{v},\uno)
=
(t{v},\uno) \bigstar  (0,\exp_\star(t\widehat{\beta})).
\]

\end{theorem}

\begin{proof} The commutativity of  $(0,\widehat{\beta})$ with all $(u,0)$, $u\in \mathcal{V}(v)$ is clearly implied by (\ref{eq:condnormal}). The (constructive) proof of the remaining assertions is parallel to that presented in Section 6.2 of
\cite{words} and will not be reproduced here.
\end{proof}

\begin{remark}\em
If the vector fields $g_j$ and $f_{\ell}$ are Hamiltonian, then the change of variables is canonical symplectic and the transformed system (\ref{eq:systnormaltrans}) is Hamiltonian; this follows from the fact that, if  $g_j$ and $f_{\ell}$ belong to  a Lie subalgebra of the Lie algebra of vector fields, then, in the  proof of the theorem, all fields (respectively all changes of variables) also belong to that subalgebra (respectively to the corresponding subgroup). From (\ref{eq:Hamseries}) the Hamiltonian function of (\ref{eq:systnormaltrans}) is
\begin{equation}
\label{eq:Hamsystnormaltrans}
     \mathcal{H}_{(v,\widehat{\beta})}(X) = \sum_{j=1}^{d} v_j \, H_j(X) +
\sum_{w\in\W}\widehat{\beta}_w\, H_w(X).
   \end{equation}
\end{remark}

\begin{remark}\em For Hamiltonian problems under  Assumption~\ref{ass:ham},
for each $u\in \mathcal{V}(v)$,
 the Poisson bracket of
\begin{equation*}
  H^{u}(X) =  \sum_{j=1}^{d} u_j \, H_j(X)
\end{equation*}
and (\ref{eq:Hamsystnormaltrans}) vanishes, because $[(u,0),(0,\widehat{\beta})]=0$.
Then each
$H^u$ is
is a first integral of the  normal form system (\ref{eq:systnormaltrans}). If $\mathcal{V}(v)$ is of dimension $d'\leq d$, then (\ref{eq:systnormaltrans}) has $d'$ linearly independent first integrals, in addition to its own Hamiltonian function (\ref{eq:Hamsystnormaltrans}).
\end{remark}

For given $v$, the change of variables  $\kappa\in \G$ and the resulting $\widehat{\beta} \in\g$ in Theorem~\ref{th:theoremnormal} are, in general, not unique:

\begin{theorem}\label{th:new}
Let $\kappa\in \G$, $\widehat{\beta} \in\g$ be the elements related by (\ref{eq:normaleqn}) whose existence is guaranteed by
Theorem~\ref{th:theoremnormal} ($\xi_v\widehat{\beta} = 0$). Then  another change of variables $\tilde \kappa\in\G$ leads to an element $\widehat{\tilde \beta}\in \g$ with
$\xi_v \widehat{\tilde \beta}=0$ if and only if, after defining
$\delta = \tilde \kappa \star \kappa^{-1}$, the relations
\begin{equation}\label{eq:saturday}
\xi_{v} \delta=0, \quad   \widehat{\tilde \beta} = \delta \star \widehat{\beta} \star \delta^{-1}
\end{equation}
are satisfied.
\end{theorem}
\begin{proof} For the \lq if\rq\ part, recall that $\xi_v$ is a derivation and that then
$$
\xi_v \delta^{-1} = -\delta^{-1}\star (\xi_v\delta)*\delta^{-1}
$$
 and
$$\xi_v (\delta   \star \widehat{\beta} \star \delta^{-1}) =
(\xi_v\delta)\star \widehat{\beta} \star \delta^{-1}+
\delta\star (\xi_v\widehat{\beta}) \star \delta^{-1}+
\delta\star \widehat{\beta} \star (\xi_v\delta^{-1}).
$$
The last expression vanishes if $\xi_v\delta = 0$ and $\xi_v \widehat{\beta}=0$. It remains to show that
$\delta \star \widehat{\beta} \star \delta^{-1}$ is the element associated to $\tilde\kappa$ as in (\ref{eq:normaleqn}). To achieve that goal it is enough to multiply (\ref{eq:normaleqn}) by $\delta$ on the right and $\delta^{-1}$ on the left and use again that $\xi_v$ is a derivation.

For the (more difficult) \lq only if\rq\ part, we begin by noting that
the map
$$
\kappa \mapsto  \kappa *\beta*\kappa^{-1} - (\xi_{v} \kappa)*\kappa^{-1}
$$
is  an action of the group $\G$ on $\g$. From there
we obtain
the homological equation
$$
\xi_v\delta= \delta \star \widehat{\beta} - \widehat{\tilde \beta} \star\delta.
$$
It is then clear that the proof will be ready if we show that $\xi_v \delta = 0$.
Since by assumption $\xi_v \widehat{\beta}=0$, $\xi_v\widehat{\tilde \beta}=0$,  the homological equation implies
$$
\xi_v (\xi_v\delta)= (\xi_v \delta) \star \widehat{\beta} - \widehat{\tilde \beta} \star(\xi_v\delta),
$$
ie
$$
(\nu_{\ell_1\cdots\ell_n}^v)^2 \delta_{\ell_1\cdots\ell_n} = \sum_{j=1}^{n-1} \Big(
(\xi_v \delta)_{\ell_1 \cdots \ell_j} \widehat{\beta}_{\ell_{j+1} \cdots \ell_{n}} -
\widehat{\tilde \beta}_{\ell_1 \cdots \ell_j} (\xi_v \delta)_{\ell_{j+1} \cdots \ell_{n}},
\Big)
$$
for each nonempty word $\ell_1\cdots\ell_n$.
If we assume inductively that $(\xi_v \delta)_w= 0$ for words with less than $n$ letters, then  $(\xi_v \delta)_{\ell_1\cdots\ell_n} =
\nu^v_{\ell_1\cdots\ell_n}\delta_{\ell_1\cdots\ell_n}=0$, because when $\nu_{\ell_1\cdots\ell_n}^v\neq 0$ the last displayed formula shows tha
$\delta_{\ell_1\cdots\ell_n}=0$.
\end{proof}

\begin{remark}\label{rem:inv}\em
The relations (\ref{eq:saturday}) may be rewritten as $\kappa^{-1}\star \widehat{\beta}\star\kappa = \tilde\kappa^{-1}\star \widehat{\tilde\beta}\star\tilde\kappa$
and $\kappa^{-1} \star \xi_{v} \kappa = \tilde\kappa^{-1} \star \xi_{v}\tilde \kappa$. Note also that, for each $u\in\mathcal{V}(v)$,
$\xi_v \delta = 0$ implies that $\xi_u \delta = 0$ and hence $\kappa^{-1} \star \xi_{u} \kappa = \tilde\kappa^{-1} \star \xi_{u}\tilde \kappa$.
\end{remark}

\begin{remark}\em If for all nonempty words $\nu_w^v \neq 0$, then $\xi_v\widehat{\beta} = 0$ and
$\widehat{\beta}_\emptyset = 0$ lead to $\widehat{\beta} = 0$.
In
Theorem~\ref{th:new},
$\delta_\emptyset = 1$ and $\xi_v\delta = 0$ imply
$\delta = \uno$ or $\kappa = \tilde \kappa$; therefore  the normal form coincides with the unperturbed problem and the change of variables $\kappa$ is unique.
\end{remark}

\subsection{Back  into  the original variables}
\label{ss:back}

Let us now push forward the pairwise commuting vector fields $g^{v}(X)$, $g^u(X)$, $u \in \mathcal{V}(v)$, and ${W}_{\widehat \beta}(X)$
to the original variables $x$.
After applying the recipe for changing variables given in Section~\ref{sss:eqnsoverlineG} and Proposition~\ref{l:LieBracketInvariance}, we obtain commuting fields
$g^{v}(x) +  W_{\kappa^{-1} \star \xi_{v} \kappa}(x)$, $g^{u}(x) +  W_{\kappa^{-1} \star \xi_{u} \kappa}(x)$, $u \in \mathcal{V}(v)$, and
$W_{\kappa^{-1}\star \widehat{\beta}\star\kappa}(x)$.
In particular we may rewrite the right hand-side of (\ref{eq:systnormal}) as a commuting sum of two terms:
\begin{equation}\label{eq:desc}
\frac{d}{dt}x = g^{v}(X) +  W_\beta(X) = \overbrace{g^{v}(x) +  W_{\kappa^{-1} \star \xi_{v} \kappa}(x)}  + \overbrace{W_{\kappa^{-1}\star \widehat{\beta}\star\kappa}(x)}.
\end{equation}
Obviously both terms in braces commute with their sum $g^v(x)+W_{\beta}(x)$. The first vector field in braces in (\ref{eq:desc}) generates a flow
\begin{equation*}
  x(t) =
\overline{W}_{(0,\kappa^{-1})\bigstar(t\omega,\uno)\bigstar(0,\kappa)}(x(0))=
\overline{W}_{(t\omega,\kappa^{-1}
\star \Xi_{t \omega} \kappa)}(x(0))
\end{equation*}
conjugate to the flow $\varphi_{tv}(x_0)= \overline{W}_{(t\omega,\uno)}(x(0))$ of the unperturbed  original system in (\ref{eq:systnormal}).
The second term in braces  generates a flow
$$x(t) = W_{\kappa^{-1}\star\exp_\star(t\widehat{\beta})*\kappa}(x(0));$$
this second term is necessary whenever the solution flows of (\ref{eq:systnormal}) and the unperturbed system cannot be conjugated to one another (for instance if the unperturbed problem has periodic solutions whose period changes after the perturbation).

 According to Remark~\ref{rem:inv},   the $v$-dependent elements $\overline{\beta} = \kappa^{-1}\star \widehat{\beta}\star\kappa$ and
$\rho(u)=\kappa^{-1} \star \xi_{u} \kappa$ do not depend on the choice of $\kappa \in \G$ in Theorem~\ref{th:theoremnormal}. To summarize, the original field  $g^v(x)+W_{\beta}(x)$ has been decomposed as a commuting sum of $g^{v}(x) + W_{\rho(v)}(x)$ and $W_{\overline{\beta}}(x)$ in such a way that
$W_{\overline{\beta}}(x)$ and $g^{v}(x) + W_{\rho(v)}(x)$ also commute with all fields $g^{u}(x) + W_{\rho(u)}(x)$, $u \in \mathcal{V}(v)$, which in turn commute among themselves. We shall present below an algorithm for computing the coefficients $\overline{\beta}$ and $\rho(u)$ needed to write these fields.

\begin{remark}\em
In the Hamiltonian case, under Assumption~\ref{ass:ham}, for each $u \in \mathcal{V}(v)$,
\begin{equation*}
  %\label{eq:ffiv}
       \mathcal{H}_{(u,\rho(u))}(x) = \sum_{j=1}^{d} u_j \, H_j(x) +
\sum_{w\in\W}\rho(u)_w\, H_w(x)
\end{equation*}
is a  first integral of the system (\ref{eq:systnormal}). Such formal  integral depends linearly on $u$ and therefore (\ref{eq:systnormal}) possesses  $d'\leq d$ linearly independent  invariants ($d'$ is the dimension of  $\mathcal{V}(v)$) in addition to its own  Hamiltonian function $\mathcal{H}_{(v,\beta)}(x)$. These invariants may be written down easily using the algorithm presented below to compute $\overline{\beta}$ and $\rho(u)$.
\end{remark}

The next results provides, under an additional hypothesis, a characterization of $\overline{\beta}$ and $\rho(u)$. In the context of Example 3 in Section~\ref{ss:examples}, the word ${\bf 0}\in\Z^d$ has
$\nu_{\zero}^{v}=0$ and the hypothesis holds whenever the Fourier coefficient $\hat f_{\bf 0}(y)$ is not identically zero.

\begin{theorem}
\label{th:uniquelydetermined}
Given $(v,\beta)\in \overline{\g}$ assume that there is a letter $\zero \in A$ such that $\nu_{\zero}^{v}=0$ and  $\beta_{\zero} \neq 0$.  Then for each $u\in \mathcal{V}(v)$, the element $\rho(u) \in \g$ is uniquely determined by the relations
\begin{gather}
  \label{eq:condrhov1}
  [(v,\beta),(u,\rho(u))]= \xi_{v} \rho(u) - \xi_{u}\beta + [\beta,\rho(u)]=0, \\
  \label{eq:condrhov2}
\rho(u)_{\ell_1 \cdots \ell_n}=0 \quad \mbox{if} \quad \nu_{\ell_1}^{v}=\cdots=\nu_{v}^{v}=0,
\end{gather}
and $\overline{\beta} \in \g$ is uniquely determined by the relations
\begin{gather}
  \label{eq:condbbeta1}
 [(v,\beta),(0,\overline{\beta})] = \xi_{v} \overline{\beta} + [\beta,\overline{\beta}]=0, \\
  \label{eq:condbbeta2}
\overline{\beta}_{\ell_1 \cdots \ell_n}=\beta_{\ell_1 \cdots \ell_n} \quad \mbox{if} \quad \nu_{\ell_1}^{v}=\cdots=\nu_{\ell_n}^{v}=0.
\end{gather}
\end{theorem}
\begin{proof}
We already know that (\ref{eq:condrhov1}) and  (\ref{eq:condbbeta1}) are fulfilled. Induction on $n$ can be used to prove (\ref{eq:condrhov2}). The relation  (\ref{eq:condbbeta2}) is implied by $\overline{\beta} = \beta - \rho(v)$.

We next show that (\ref{eq:condbbeta1})--(\ref{eq:condbbeta2}) uniquely determine $\overline{\beta} \in \g$. (That (\ref{eq:condrhov1})--(\ref{eq:condrhov2}) uniquely determine $\rho(v)$ is proved in a very similar way.) We work by induction on the number of letters.
The condition (\ref{eq:condbbeta1}) for words with one letter $\ell \in A$ yields $\nu_{\ell}^{v}\overline{\beta}_\ell=0$. Therefore, if $\nu_{\ell}^{v}\neq 0$, then $\overline{\beta}_{\ell} = 0$;  otherwise, $\overline{\beta}_{\ell}=\beta_{\ell}$ by virtue of (\ref{eq:condbbeta2}).   For a   word $w=\ell_1\cdots \ell_n$, $n>1$,
(\ref{eq:condbbeta1}) reads
\begin{equation*}
 \nu_{w}^{v} \overline{\beta}_w -
\sum_{k=1}^{n-1} (\beta_{\ell_1\cdots \ell_{k}}\overline{\beta}_{\ell_{k+1}\cdots \ell_n} -
                 \overline{\beta}_{\ell_1\cdots \ell_{k}}\beta_{\ell_{k+1}\cdots \ell_n})=0
\end{equation*}
a relation that obviously determines
$\overline{\beta}_{w}$ when $\nu_{w}^{v} \neq 0$. If $\nu_{w}^{v} = 0$,  we distinguish two cases. If $\nu_{\ell_1}^{v}=\dots=\nu_{\ell_n}^{v}=0$  then $\overline{\beta}_{\ell_1\cdots \ell_n}$ is determined by condition (\ref{eq:condbbeta2}). Otherwise, we consider (\ref{eq:condbbeta1}) for the word $w=\zero\ell_1\cdots \ell_n$, to obtain
\begin{equation*}
\overline{\beta}_{\ell_1\cdots \ell_n} - \frac{\beta_{\ell_n}}{\beta_{\zero}}\overline{\beta}_{\zero\ell_1 \cdots \ell_{n-1}} = \cdots
\end{equation*}
where the right-hand side is, by the induction hypothesis, uniquely determined, but $\overline{\beta}_{\zero\ell_1 \cdots \ell_{n-1}}$ may in principle pose difficulties because it refers to a word with $n$ letters. If $\nu_{\ell_n}^{v} = 0$, then $\nu_{\zero\ell_1 \cdots \ell_{n-1}}^{v} = 0$ so that $\overline{\beta}_{\zero\ell_1 \cdots \ell_{n-1}}$ cannot be found by means of (\ref{eq:condbbeta1}). But we may then repeat the process until we eventually obtain
\begin{equation*}
%\label{eq:thaux}
\overline{\beta}_{\ell_1\cdots \ell_n} - \frac{\beta_{\ell_n}\cdots \beta_{\ell_{k+1}}}{\beta_{\zero}^{n-k+1}}\overline{\beta}_{\zero\cdots\zero \ell_{1} \cdots \ell_{k}} = \cdots
\end{equation*}
with  $\nu_{\zero\cdots\zero \ell_{1} \cdots \ell_{k}}^{v} \neq 0$, and hence
$\overline{\beta}_{\zero\cdots\zero \ell_{1} \cdots \ell_{k}}$ determined via (\ref{eq:condbbeta1}) .
\end{proof}
% \nu_{w} \overline{\beta}_w - \beta_{\ell_1}\overline{\beta}_{\ell_2\cdots \ell_n} +  \overline{\beta}_{\ell_1\cdots \ell_{n-1}} \beta_{\ell_n} = \sum_{k=2}^{n-1} (\beta_{\ell_1}\overline{\beta}_{\ell_{k+1}\cdots \ell_n} )
\begin{remark}\em
  The proof of Theorem~\ref{th:uniquelydetermined} provides very simple recursions for the coefficients of $\overline{\beta}$ and $\rho(v)$ corresponding to the original problem (\ref{eq:odegf})--(\ref{eq:fdec}) where $\beta \in \g$ is given by (\ref{eq:b}):
\begin{gather*}
    \overline{\beta}_{\ell} = 1, \quad \mbox{if} \quad \nu_{\ell}^{v}=0, \\
    \overline{\beta}_{\ell} = 0, \quad \mbox{if} \quad \nu_{\ell}^{v}\neq 0, \\
    \overline{\beta}_{\ell_1\cdots \ell_n} = \frac{\overline{\beta}_{\ell_1\cdots \ell_{n-1}}-\overline{\beta}_{\ell_2\cdots \ell_n}}{\nu_{\ell_1\cdots \ell_n}^{v}} , \quad \mbox{if} \quad n>1, \ \nu_{\ell_1\cdots \ell_n}^{v} \neq 0,
\end{gather*}
and if $n>1$ and $\nu_{\ell_1\cdots \ell_n}^{v} =0$,
\begin{equation*}
\overline{\beta}_{\ell_1\cdots \ell_n} =
\left\{
\begin{array}{cl}
 0, & \mbox{if} \quad \nu_{\ell_1}^{v} = \cdots = \nu_{\ell_n}^{v} = 0, \\
\overline{\beta}_{\zero \ell_1\cdots \ell_{n-1}}, & \mbox{otherwise.} \\
\end{array}
\right.
\end{equation*}
The recursions for $\rho(u)$, $u \in \mathcal{V}(v)$ are obtained by replacing the first two lines above by
 \begin{gather*}
\rho(u)_{\ell}=0, \quad \mbox{if} \quad \nu_{\ell}^{v}=0, \\
\rho(u)_{\ell}=\frac{\nu_{\ell}^{u}}{\nu_{\ell}^{v}}, \quad \mbox{if} \quad \nu_{\ell}^{v}\neq 0;
 \end{gather*}
and the last three lines  remain the same, with $\rho(u)$ in lieu of $\overline{\beta}$.
\end{remark}

\subsection{Normal forms for elements in the group $\overline{\G}$}

Given an element $(v,\eta) \in \overline{\G}$, it is sometimes desirable (in particular, when $(v,\eta)$ represents an integrator for the system  (\ref{eq:systnormal})) to reduce it to normal form. The following developments are parallel to those presented above for elements in the algebra; full details will not be given.
Assume that we study the iteration
\begin{equation*}
x_n = \overline{W}_{(v,\eta)}(x_{n-1}), \quad n=1,2,3,\ldots
\end{equation*}
A change of variables $x_n=W_{\kappa}(X_n)$, $\kappa \in \G$, gives
\begin{equation*}
X_n = \overline{W}_{(v,\widehat{\eta})}(X_{n-1})=\varphi_{v}({W}_{\widehat{\eta}}(X_{n-1})), \quad n=1,2,3,\ldots,
\end{equation*}
with
\begin{equation}
\label{eq:hdelta}
(v,\widehat{\eta}) = (0,\kappa) \bigstar (v,\eta) \bigstar (0,\kappa)^{-1} =
(v, \kappa \star \eta \star (\Xi_{v} \kappa^{-1})) \in \overline{\G}.
%(v,\widehat{\delta}) = (0,\kappa) \bigstar (v,\delta) \bigstar (0,\kappa)^{-1} =
%(v, \kappa \star \delta \star \Xi_{v} \kappa^{-1}) \in \overline{\G}
\end{equation}
If we find $\kappa \in \G$ such that $(0,\widehat{\eta}) \in \overline{\g}$ commutes with $(v,0) \in \overline{\g}$, then  $x_n = W_{\kappa}(\varphi_{n v} (z_n)))$, where $z_n$ satisfies the simpler (normal form) recursion
\begin{equation*}
z_n = {W}_{\widehat{\eta}}(z_{n-1}), \quad n=1,2,3,\ldots,
\end{equation*}
with $z_0 = X_0 = W_{\kappa^{-1}}(x_0)$.

Since $(v,0) \star (0,\widehat{\eta}) = (v,\Xi_v \widehat{\eta})$ and
$(0,\widehat{\eta}) \star (v,0)=(v,\widehat{\eta})$, we have that $(0,\widehat{\eta})$ and $(v,0)$ commute if and only if
\begin{equation}
\label{eq:v-commuting-G}
\Xi_v \widehat{\eta} = \widehat{\eta}.
\end{equation}

Moreover,  if $u \in \C^d$ satisfies
\begin{equation}
\label{eq:condvG}
   \exp{(\nu^{u}_{w})} =1 \quad \mbox{for all} \quad w \in \W \quad \mbox{such that} \quad
\exp{(\nu^{v}_{w})} = 1,
\end{equation}
then $\Xi_u \widehat{\eta} = \widehat{\eta}$, and thus
\begin{equation*}
  (u,\uno)\bigstar (0,\widehat{\eta}) = (0,\widehat{\eta}) \bigstar (u,\uno).
\end{equation*}

The following analogue  of Theorem~\ref{th:theoremnormal} holds.

\begin{theorem}
\label{th:normal-forms-in-G}
Given $(v,\eta) \in \overline{\G}$, there exists  $\kappa\in\G$ such that  $(v,\widehat{\eta}) \in \overline{\G}$ given by (\ref{eq:hdelta}) satisfies (\ref{eq:v-commuting-G}).
%(i.e., $ \widehat{\eta}_w = 0$ for all words $w \in \W$ with $\exp{(\nu^{v}_{w})} \neq 1$).
%
%Such $\kappa\in \G$ can be uniquely determined by imposing the additional condition $(\log_{\star}\kappa)_w=0\ $\ for all words $w$ with $\exp{(\nu^{v}_{w})} =1$.
Hence, the map $\overline{W}_{(v,\eta)}$ is conjugate to
\begin{equation*}
\overline{W}_{(v,\widehat{\eta})} = \varphi_{v} \circ W_{\widehat{\eta}} = W_{\widehat{\eta}} \circ\varphi_{v}
\end{equation*}
 through the transformation $x = W_\kappa(X)$.
Also, $\varphi_{u}$ and ${W}_{\widehat \eta}$ commute for all $u \in \C^d$ satisfying (\ref{eq:condvG}).
\end{theorem}

As we now discuss, this normal form result for elements in the group $\overline{\G}$ is actually related to the analogous result  for elements in the Lie algebra $\overline{\g}$.
Under the assumption (\ref{eq:non-resonance-for-log}), the element $(v,\eta) \in \overline{\G}$ is the exponential of an element $(v,\beta) \in \overline{\g}$. The application of Theorem~\ref{th:theoremnormal}
to the logarithm $(v,\beta)$ yields an element $\kappa \in \G$  (which is used to change variables) and an element  $\widehat \beta \in \g$ (which appears in the normal form).
Then  $(v,\widehat{\eta}) \in \overline{\G}$ in the statement of Theorem~\ref{th:normal-forms-in-G}  can be taken to  coincide with  the exponential of $(v,\widehat{\beta})$ (ie $\widehat{\eta}$ can be chosen as  $\exp_{\star} \widehat{\beta}$). And, in addition, the same $\kappa$  that changes $\beta$ into $\widehat{\beta}$ changes $\eta$ into $\widehat{\eta}$.

\begin{remark}\em
In the Hamiltonian case, for each $\eta\in\G$ the series  $W_{(v,\eta)}$ is canonical symplectic.
The transformation $W_{\kappa}$  and the conjugate map $\overline{W}_{(v,\widehat{\eta})}$ provided by the theorem are also canonical symplectic.
\end{remark}
\begin{remark}\em
 Consider the Hamiltonian case and suppose that Assumption~\ref{ass:ham} is satisfied.
If in addition condition  (\ref{eq:non-resonance-for-log}) holds true, then $\xi_v \widehat{\eta}=0$, and thus, for all $u \in \mathcal{V}(v)$,  $\xi_u \widehat{\eta} = 0$. Then, for all $u \in \mathcal{V}(v)$, $H^u(X)$ is an invariant for the  map $W_{\widehat{\eta}}$ (and hence, also for the map  $\overline{W}_{(v,\widehat{\eta})} = \varphi_v  \circ W_{\widehat{\eta}}$) because
\begin{align*}
H^{u}(\overline{W}_{(0,\widehat{\eta})}(X))
&= \mathcal{H}_{(u,0)}(\overline{W}_{(0,\widehat{\eta})}(X)) \\
&= \mathcal{H}_{(u,-(\xi_u \widehat{\eta})\star \widehat{\eta}^{-1})}(X) \\
&= \mathcal{H}_{(u,0)}(X) \\
&= H^u(X).
\end{align*}
\end{remark}

The pair $(\kappa,\widehat{\eta})$ in the statement of Theorem~\ref{th:normal-forms-in-G} is not unique. The freedom in the choice of this pair is described in our next result.

\begin{theorem}
\label{th:freedom-of-kappa-in-G}
Let $ \eta \in \G$, and assume that $(\kappa,\widehat{\eta}) \in \G \times \G$ satisfies (\ref{eq:hdelta})-- (\ref{eq:v-commuting-G}). Given  $(\widetilde{\kappa}, \widehat{\widetilde{\eta}})\in \G \times \G$ such that
\begin{equation*}
(v,\widehat{\widetilde{\eta}}) = (0,\widetilde{\kappa}) \bigstar (v,\eta) \bigstar (0,\widetilde{\kappa})^{-1} =
(v, \widetilde{\kappa} \star \eta \star (\Xi_{v} \widetilde{\kappa}^{-1})) \in \overline{\G},
\end{equation*}
then $\Xi_v \widehat{\widetilde{\eta}} = \widehat{\widetilde{\eta}}$ if and only if $\Xi_v \delta = \delta$, where $\delta = \widetilde{\kappa} \star \kappa^{-1}$, and in that case, $\widehat{\widetilde{\eta}} = \delta \star \widehat{\eta} \star \delta^{-1}$.
\end{theorem}

Back into the original variables, one obtains a decomposition of $(v,\eta)$ as the product of two commuting elements in $\overline{\G}$, namely,
\begin{equation*}
 (0,\kappa)^{-1} \bigstar (v,\uno) \bigstar (0,\kappa) =
(v, \kappa^{-1}  \star (\Xi_{v} \kappa))  \in \overline{\G}
\end{equation*}
and
\begin{equation*}
 (0,\kappa)^{-1} \bigstar (0,\widehat{\eta}) \bigstar (0,\kappa) =
(v, \kappa^{-1}  \star \widehat{\eta} \star \kappa) \in \G,
\end{equation*}
which provides a representation of the original map $\overline{W}_{(v,\eta)}$ as the composition of two commuting maps $\overline{W}_{(v,\kappa^{-1}  \star (\Xi_{v} \kappa)) )}$ and $W_{\kappa^{-1}  \star \widehat{\eta} \star \kappa}$. Theorem~\ref{th:freedom-of-kappa-in-G} implies that both $ \kappa^{-1}  \star (\Xi_{v} \kappa)$ and $\overline{\eta}=\kappa^{-1}  \star \widehat{\eta} \star \kappa)$ are independent of the choice of $\kappa$.

We finally have:
\begin{remark}\em
 For Hamiltonian problems satisfying Assumption~\ref{ass:ham} and under the additional non-resonance condition (\ref{eq:non-resonance-for-log}),   the quantity
$
\mathcal{H}_{(u,\kappa^{-1} \star (\xi_{u} \kappa))}(x)
$
is, for each $u \in  \mathcal{V}(v)$, a formal invariant of the map $\overline{W}_{(v,\eta)}$.
In fact,
from the last remark we know that $H^u(X)$ is an invariant for the map $X \mapsto \overline{W}_{(v,\widehat{\eta})}(X)$ and thus
\begin{align*}
H^u(W_{\kappa^{-1}}(x)) &= \mathcal{H}_{(u,0)}(\overline{W}_{(0,\kappa^{-1})}(x)) \\
&=  \mathcal{H}_{(u,-(\xi_u \kappa^{-1}) \star \kappa))}(x) \\
&=  \mathcal{H}_{(u,\kappa^{-1} \star \xi_u (\kappa)))}(x),
\end{align*}
is an invariant for the map $x \mapsto \overline{W}_{(v,\eta)}(x)$, where $x=W_{\kappa}(X)$.
\end{remark}

\bigskip

{\bf Acknowledgement}. A. Murua and J.M.
Sanz-Serna have been supported by proj\-ects MTM2013-46553-C3-2-P and MTM2013-46553-C3-1-P from Ministerio de Eco\-nom\'{\i}a y Comercio, Spain. Additionally A. Murua has been partially supported by the Basque Government  (Consolidated Research Group IT649-13).


\begin{thebibliography}{99}
 \bibitem{arnoldode}{\sc V. I. Arnold}, {\em Geometrical Methods in the Theory of Ordinary Differential
        Equations, 2nd ed.}, Springer, New York, 1988.
 \bibitem{arnoldmec}{\sc V. I. Arnold,} {\em Mathematical Methods of Classical Mechanics, 2nd ed.,} Springer, New York, 1989.
 \bibitem{geir}{ \sc G. Bogfjellmo and A. Schmeding}, {\em The Lie group structure of the Butcher group},
  Found. Comput. Math., submitted.
 \bibitem{brouder}{\sc Ch. Brouder}, {\em Trees, renormalization and differential equations}, BIT Numerical Mathematics,  44 (2004), pp.~425--438.
 \bibitem{part1}{\sc P. Chartier, A. Murua, and J.M. Sanz-Serna}, {\em Higher-Order averaging, formal series and numerical integration I:
     B-series}, Found. Comput. Math.  10  (2010), pp.~695--727.
\bibitem{part2}{\sc P. Chartier, A. Murua, and J.M. Sanz-Serna}, {\em Higher-Order averaging,
        formal series and numerical integration II: the quasi-periodic
        case},   Found. Comput. Math.,  12   (2012), pp.~471-508.
\bibitem{orlando}{\sc P. Chartier, A. Murua, and J.M. Sanz-Serna}, {\em A formal
    series approach to averaging: exponentially small error estimates},
     DCDS A,  32 (2012), pp.~3009-3027.
\bibitem{part3}{\sc P. Chartier, A. Murua, and J.M. Sanz-Serna}, {\em Higher-Order averaging,
        formal series and numerical integration II}, Found. Comput. Math.,  (2013) DOI 10.1007/s10208-013-9175-7.
 \bibitem{fm} {\sc F. Fauvet and F. Menous,} {\em Ecalle's arborification-coarborification transforms and Connes-Kreimer Hopf algebra}, arXiv; 1212.4740v2.
 \bibitem{hlw}{\sc E. Hairer, Ch. Lubich, and G. Wanner}, {\em Geometric Numerical Integration, 2nd
        ed.,} Springer, Berlin, 2006.
 \bibitem{HW} {\sc E. Hairer and G. Wanner}, {\em On the Butcher group and general multi-value methods}, Computing, 13 (1974), pp.~1--15.
 \bibitem{anderfocm}{\sc A. Murua,} {\em The Hopf algebra of rooted trees, free Lie algebras and Lie series,}
     { Found. Comput. Math.},  6  (2006), pp.~387--426.
 \bibitem{words}{\sc A. Murua and J.M. Sanz-Serna},  {\em Word series for dynamical systems and their numerical integrators}, arXiv:1502.05528.
\bibitem{ssc}{\sc J. M. Sanz-Serna and M. P. Calvo}, {\em Numerical Hamiltonian
        Problems,} Chapman and Hall, London, 1994.
 \bibitem{china} {\sc J.M. Sanz-Serna and A. Murua}, {\em Formal series and numerical integrators: some history and some new techniques}, in Proceedings of the 8th International Congress on Industrial and Applied Mathematics (ICIAM 2015), Lei Guo and Zhi-Ming eds., Higher Edication, Press, Beijing, 2015, pp.~311--331.
 \end{thebibliography}
\end{document}